\documentclass[letterpaper, 10 pt, conference]{ieeeconf}  

\IEEEoverridecommandlockouts                              
\usepackage{amsfonts}
\usepackage{graphicx}
\usepackage[dvipsnames]{xcolor}
\usepackage{amsmath}
\usepackage{amssymb}
\usepackage{dsfont}

\usepackage{ifthen}
\usepackage{color}
\usepackage{cite}

\usepackage{mathtools}
\usepackage{booktabs}
\usepackage{url}
\usepackage{hyperref}
\usepackage[caption=false,font=footnotesize]{subfig}
\usepackage{float}
\usepackage{multirow}
\usepackage{comment}
\usepackage{tikz}
\usepackage{pgfplots}
\pgfplotsset{compat=1.12}
\usepackage{algorithm}
\usepackage{algpseudocode}

\newcommand{\R}{{\rm I\!R}}

\newtheorem{lemma}{Lemma}

\newtheorem{definition}{Definition}
\newtheorem{proposition}{Proposition}

\def\qed{ \rule{.08in}{.08in}}

\newcommand{\bestdelta}[1]{\textbf{\textcolor{red}{#1}}}
\newcommand{\bestwin}[1]{\textbf{\textcolor{blue}{#1}}}
\newcommand{\besttime}[1]{\textbf{\textcolor{red}{#1}}}

\title{\LARGE \bf Toward Maximizing Algebraic Connectivity via Density-Dependent Backbone Construction and Reinforcement Learning Completion
}

\author{Zeru Zhu \hspace{.3in} Marco Gamarra \hspace{.3in} Ji Liu
\thanks{Z. Zhu is with the Department of Applied Mathematics and Statistics at Stony Brook University.
M. Gamarra is with the Air Force Research Laboratory.
J.~Liu is with the Department of Electrical and Computer Engineering at Stony Brook University.
Email addresses: 
\href{mailto:zeru.zhu@stonybrook.edu}{zeru.zhu@stonybrook.edu},
\href{mailto:marco.gamarra@us.af.mil}{marco.gamarra@us.af.mil},
\href{mailto:ji.liu@stonybrook.edu}{ji.liu@stonybrook.edu}.
}
}

\begin{document}

\maketitle
\thispagestyle{empty}
\pagestyle{empty}

\begin{abstract}
This paper studies an open network design problem of constructing graph topologies with maximum algebraic connectivity under fixed numbers of vertices and edges. A density-dependent approach is proposed that decomposes the design task into two stages: a high-quality, low-cost structural backbone construction and a reinforcement learning completion stage. The method selects the backbone based on density, using complete multipartite graphs in dense regimes and Cayley graphs in low-to-intermediate regimes, and then refines the remaining edges through a learning-based procedure. Numerical results show that the proposed method achieves high algebraic connectivity with favorable runtime scalability, outperforming existing methods. These results support a structural-family-first strategy for scalable network design.
\end{abstract}

\section{Introduction}

Over the past two decades, consensus has been extensively studied and has achieved notable success \cite{vicsekmodel,reza1,ReBe05}, finding applications in a broad range of distributed control and computation problems \cite{survey,nedic2009distributed,tacle,split}.
Consensus processes are fundamental to almost all cooperative multi-agent algorithms and, in many cases, significantly influence or even determine their convergence rate.

A typical continuous-time linear consensus process among $n$ agents over a connected simple graph $G$ can be modeled by a linear differential equation of the form $\dot x(t)=-Lx(t)$, where $x(t)$ is the state vector in $\R^n$ and $L$ is the ``Laplacian matrix'' of $G$. 
The $i$th entry of $x(t)$ represents the state of agent $i$. 
The multi-agent system is said to reach a consensus if all $x_i(t)$ are equal. 
For a simple graph $G$ with $n$ vertices, its Laplacian matrix is defined as $L=D-A$, where $D$ and $A$ denote its degree matrix and adjacency matrix, respectively. 
Here, $D$ is an $n\times n$ diagonal matrix whose $i$th diagonal entry equals the degree of vertex $i$, and $A$ is an $n\times n$ matrix whose $ij$th entry equals 1 if $(i,j)$ is an edge in $G$ and otherwise equals 0. 
It is well known that $L$ is positive semidefinite with smallest eigenvalue $0$, and its second smallest eigenvalue, called the algebraic connectivity of $G$, 
is positive if and only if $G$ is connected \cite{Fiedler73}. 
It has been shown that $\dot x(t)=-Lx(t)$ will reach a consensus if and only if $G$ is connected. 
More importantly, the (worst-case) convergence rate of consensus is characterized by the algebraic connectivity, with larger values leading to faster consensus \cite{reza1}. 

Motivated by the above discussion, a natural and fundamental problem is to design network topologies that achieve fast, or sometimes fastest, consensus. This problem has been extensively studied for a long time, among others \cite{moura,ming,bayen,dai,dong,ogiwara}. Despite these efforts, the following question remains largely open: Given fixed numbers of vertices and edges, which graphs maximize the algebraic connectivity?  
To date, this question has only been theoretically answered in certain special cases \cite{ogiwara,Shahbaz23}.
Beyond the challenges in theoretical analysis, the question is also computationally intractable in numerical experiments, especially for large-scale graphs. 
The related edge-augmentation problem is NP-hard \cite{MOSKAOYAMA2008677}.
Even though exhaustive combinatorial search is conceptually possible, its super-exponential growth makes it computationally impractical beyond very small graph sizes. The structural characterization of optimal graphs that maximize algebraic connectivity therefore remains poorly understood, let alone the associated computational complexity.

\subsection{Related Work and Benchmark Methods}

Due to the theoretical and computational hurdles, existing literature includes practical heuristics (typically without guarantees) and exact formulations (typically with limited scalability), such as MISDP-based approaches.
Constructing graphs from scratch under fixed vertex/edge budgets has received less attention: \cite{kurahashi2024depth} is restricted to special regular graphs, 
and \cite{somisetty2025optimal} studies a weighted MISDP formulation with outer-approximation procedures.

Most existing works investigate a closely related augmentation problem: given a connected graph, how to add a fixed number of edges to maximally increase algebraic connectivity. This problem is NP-hard \cite{MOSKAOYAMA2008677}. Prior work includes SDP relaxations for Boolean edge-addition formulations \cite{ghosh2006growing}, as well as MISDP formulations \cite{somisetty2024spectral}. A number of heuristic and approximation methods have also been proposed. Notable among them are Fiedler vector based greedy methods \cite{ghosh2006growing,yin2024accelerating}, secular equation (for the Laplacian under a rank-one perturbation) guided edge selection \cite{kim2010bisection}, tabu search \cite{wei2014algebraic}, and Frank-Wolfe optimization with Fiedler vector gradients \cite{9981584}. In principle, these methods can be extended to construct a graph from scratch via different strategies, e.g., applying them iteratively.


These methods are computationally expensive and do not scale well to large graph construction.
We benchmark representative methods in \cite{ghosh2006growing,8412478,wei2014algebraic,somisetty2024spectral,9981584,somisetty2025optimal} for medium-sized graphs from scratch; results indicate that Fiedler vector based greedy (FVG) \cite{ghosh2006growing} gives the best practical quality-runtime trade-off among full-construction baselines (Section~\ref{subsec:benchmark}), so we adopt it as the primary benchmark. We do not benchmark \cite{yin2024accelerating} since it is a decentralized implementation of FVG under a different problem setting. We also exclude \cite{kim2010bisection} in the main benchmark because it is an exact $\Delta \lambda_2$ greedy algorithm; in our setting it has higher cost and does not improve the quality-runtime trade-off relative to FVG.

For a graph, any eigenvector corresponding to the second smallest eigenvalue of its Laplacian matrix is called a Fiedler vector of the graph \cite{Fiedler73}. We apply the Fiedler vector based greedy method to construct a graph from scratch by iteratively adding edges that maximize the increase in algebraic connectivity. Specifically, at each iteration, the increase caused by adding an edge $(i,j)$ is approximated via first-order eigenvalue perturbation as $(v_i - v_j)^2$, where $v$ is a selected Fiedler vector of the current graph; if $\lambda_2$ has multiplicity greater than one, $v$ is chosen from the corresponding eigenspace.

\subsection{Problem Formulation}

For a connected simple graph \(G=(V,E)\) with \(|V|=n\) and \(|E|=m\),
its Laplacian eigenvalues are ordered as \(0=\lambda_1(L(G))\le \lambda_2(L(G))\le \cdots\le \lambda_n(L(G))\).
Here \(\lambda_2(L(G))\) is the algebraic connectivity. Let \(\mathcal G_n\) denote the set of all connected simple graphs with $n$ vertices. 
For fixed \((n,m)\), define \(\mathcal G_{n,m}=\{G\in\mathcal G_n: |E|=m\}\).
The network design problem of interest is
\begin{equation}
\max_{G\in\mathcal G_{n,m}}\ \lambda_2(L(G)).
\label{eq:main_opt}
\end{equation}
We refer to \eqref{eq:main_opt} as algebraic connectivity maximization.
Given the combinatorial search space, the goal is to construct a graph with algebraic connectivity as large as possible for fixed \(n\) and \(m\).

\subsection{Contributions}
This paper proposes a graph theory guided reinforcement learning approach to the algebraic connectivity maximization problem. The approach decomposes the network design into two stages. First, we develop a density-dependent, high-quality, low-cost backbone stage for fixed-\((n,m)\) algebraic-connectivity maximization, using complete multipartite and Cayley graph families with routing and fallback (cf. Section~\ref{sec:backbone}).
Second, we formulate completion as a residual edge-allocation problem from the selected backbone and use reinforcement learning (RL) to optimize only this sequential residual stage (cf. Section \ref{sec:RL}).
Together, this backbone-plus-completion decomposition gives the best quality/runtime scaling among tested baselines as \(n\) increases, outperforms all existing methods, and thus supports a structural-family-first viewpoint for scalable algebraic-connectivity maximization (cf. Section \ref{sec:simulation}).

\section{Density-Dependent Backbone Construction} \label{sec:backbone}

Most existing methods in related work assume a given initial graph and then optimize edge additions.
Our approach treats initialization as a design variable that strongly affects both remaining-edge search difficulty and final graph quality. Meanwhile, simulations show that graphs with high algebraic connectivity exhibit distinct structures depending on graph density. 
This motivates a density-dependent structural initialization that chooses the initial graph family according to edge density.
To parameterize edge budgets uniformly across different \(n\), we use the connected-range density
\begin{equation}
\rho
=
\frac{m-(n-1)}{\binom{n}{2}-(n-1)},
\;\;\; {\rm with} \;\;\;
n-1\le m\le \binom{n}{2}.
\label{eq:rho_connected}
\end{equation}
It is easy to see that $\rho\in[0,1]$. 
\(\rho=0\) corresponds to the minimum connectivity budget (\(m=n-1\)) and \(\rho=1\) corresponds to the complete graph.
For each target \((n,m)\), we first construct a connected graph \(G_0=(V,E_0)\) with \(|E_0|=m_0\le m\), aiming for large \(\lambda_2(G_0)\) relative to its edge budget.
The remaining \(m-m_0\) edges are added in a later completion stage.
We refer to this initialized graph \(G_0\) as a \emph{backbone}.
Section II-A proposes the envelope backbone, Section II-B proposes the Cayley backbone, Section II-C presents an algorithm for selecting between the two backbones based on density (density routing), and Section II-D analyzes the algorithm complexity and scalability.

\subsection{Envelope Backbone via Complete Multipartite Graphs}
For a fixed node count \(n\), each graph in \(\mathcal G_n\) corresponds to a discrete point in the \((\rho,\lambda_2)\) plane.
Since \(\lambda_2\) is monotone nondecreasing under edge addition, increasing edge density does not reduce the best achievable \(\lambda_2\);
hence, across the density spectrum, any maximizer of \eqref{eq:main_opt} for some feasible \(m\) must lie on the weak upper-left staircase, i.e., weak-maximal points under the order ``smaller \(\rho\), larger \(\lambda_2\).'' See Fig.~\ref{fig:staircase_random_n32}. We formalize this notion as follows.

\begin{figure}[t]
\centering
\includegraphics[width=0.95\linewidth]{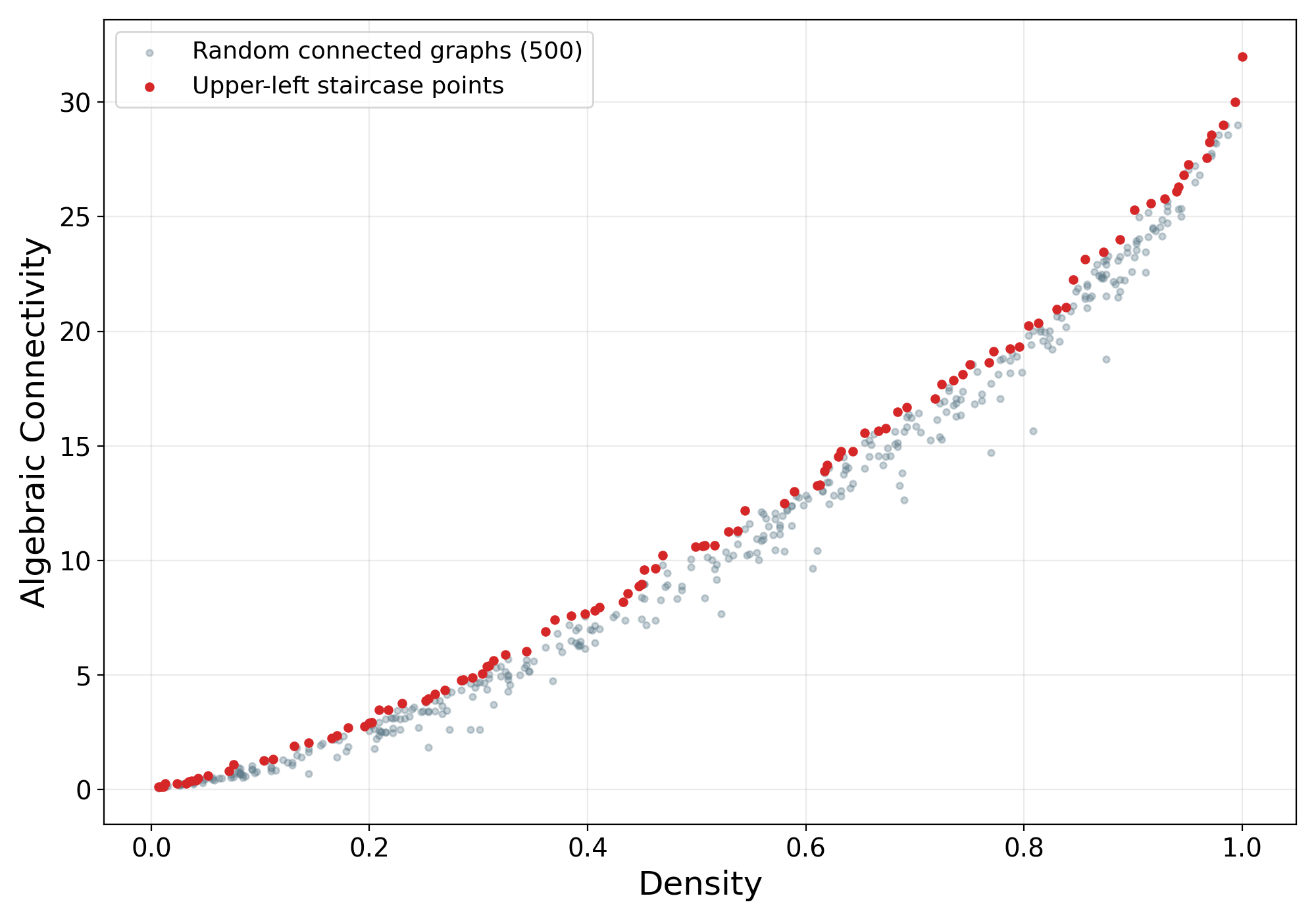}
\caption{Random connected graphs (\(n=32\)); highlighted points are weak-maximal staircase points in \((\rho,\lambda_2)\).}
\label{fig:staircase_random_n32}
\end{figure}

\begin{definition}
Let \(\mathcal P\subset\mathbb R^2\) be a finite point set, with points written as \(p=(x_p,y_p)\).
For \(p,q\in\mathcal P\), we say that \(q\) \emph{weakly dominates} \(p\) if \(x_q\le x_p\) and \(y_q>y_p\).
We say that \(q\) \emph{strictly dominates} \(p\) if \(x_q\le x_p\), \(y_q\ge y_p\), and at least one inequality is strict.
A point \(p\) is \emph{weak-maximal} (resp. \emph{strict-maximal}) if no point in \(\mathcal P\) weakly dominates (resp. strictly dominates) \(p\).
\end{definition}

In our setting, \(x=\rho\) and \(y=\lambda_2\), and \(\mathcal P\) is the set of graph-induced points at fixed \(n\). In this convention, weak dominance is strict in \(\lambda_2\) and non-strict in \(\rho\): weak-maximality characterizes where optimizers of \eqref{eq:main_opt} can lie, while strict-maximality keeps only the left boundary among points with equal \(\lambda_2\).
Characterizing these maximal sets of \(\mathcal G_n\) directly is generally intractable.
We therefore study them within a structured subfamily where edge count and \(\lambda_2\) admit closed-form expressions.
We next introduce the complete multipartite family \(\mathcal K_n\), which is known to contain graphs with strong algebraic connectivity \cite{ogiwara,Shahbaz23}.
An illustrative subset of \(\mathcal K_n\) in the \((\rho,\lambda_2)\) plane is shown in Fig.~\ref{fig:envelope_staircase}.

\begin{definition}
Let \(r\in\{2,\dots,n\}\) and \(1\le n_1\le \cdots \le n_r\) with \(\sum_{i=1}^r n_i=n\).
The complete multipartite graph \(K_{n_1,\dots,n_r}\) has all edges between different parts and no edges within any part.
Denote by \(\mathcal K_n\subseteq\mathcal G_n\) the family of all such complete multipartite graphs on \(n\) vertices.
\end{definition}

It is known that when all parts have equal size (i.e., \(r\) divides \(n\)), this graph is a maximizer of \(\lambda_2\) for the corresponding edge budget \cite{ishii2016extensions}.
We refer to this case as a balanced complete multipartite graph.

The first ingredient is a closed-form edge-count identity.

\begin{lemma}\label{lem:cmp_edge_count}
For \(G=K_{n_1,\dots,n_r}\),
\begin{equation}
m(G)=\sum_{1\le i<j\le r}n_i n_j
=\frac{1}{2}\bigg(n^2-\sum_{i=1}^r n_i^2\bigg).
\label{eq:cmp_edge}
\end{equation}
\end{lemma}

\vspace{.12in}

{\em Proof:}
Edges appear only across different parts, so \(m(G)=\sum_{1\le i<j\le r} n_i n_j\).
Using \(\left(\sum_{i=1}^r n_i\right)^2=\sum_{i=1}^r n_i^2+2\sum_{1\le i<j\le r} n_i n_j\) and \(\sum_i n_i=n\), we obtain \eqref{eq:cmp_edge}.
\hfill$\qed$

Hence edge minimization at fixed \(n\) is equivalent to maximizing \(\sum_i n_i^2\).

The second ingredient is a closed-form expression for \(\lambda_2\) in this family.

\begin{lemma}\label{lem:cmp_lambda2_formula}
Let \(G=K_{n_1,\dots,n_r}\) with \(2\le r\le n-1\) and \(n_1\le\cdots\le n_r\).
Then
\begin{equation}
\lambda_2(G)=n-n_r.
\label{eq:cmp_lambda2}
\end{equation}
For \(r=n\), \(G=K_n\) and \(\lambda_2(G)=n\).
\end{lemma}

{\em Proof:}
The lemma is a direct consequence of Proposition 3.2 in \cite{klee_stamps_2022_rank1}.
\hfill$\qed$

This shows that, within \(\mathcal K_n\), fixing \(\lambda_2=k\) is equivalent to fixing the largest part size as \(n-k\).

We now define the envelope point at a fixed \(\lambda_2\)-level.

\begin{definition}
Fix \(k\in\{1,\dots,n-2\}\).  
Among all \(G\in\mathcal K_n\) with \(\lambda_2(G)=k\), a graph with minimum edge count is called an envelope point at level \(k\).
\end{definition}

By definition, these are exactly strict-maximal points of \(\mathcal K_n\) in the \((\rho,\lambda_2)\) plane.
The next proposition gives an explicit construction and edge formula.

\begin{proposition}[Envelope Point at Level \(k\)]
\label{prop:env_formula}
Fix \(k\in\{1,\dots,n-2\}\), and write
\begin{equation}
n=q(n-k)+b,\qquad 0\le b<n-k.
\label{eq:env_division}
\end{equation}
An envelope point is obtained by taking \(q\) parts of size \(n-k\), and, if \(b>0\), one additional part of size \(b\).
Its edge count is
\begin{equation}
m_{\mathrm{env}}(k)=\frac{1}{2}\left(n^2-q(n-k)^2-b^2\right).
\label{eq:env_mk}
\end{equation}
\end{proposition}

\vspace{.12in}

{\em Proof:}
By Lemma~\ref{lem:cmp_lambda2_formula}, fixing \(\lambda_2=k\) is equivalent to fixing the largest part size as \(n_r=n-k\).
Thus feasible partitions satisfy \(n_i\le n-k\) and \(\max_i n_i=n-k\).
By Lemma~\ref{lem:cmp_edge_count}, minimizing \(m(G)\) is equivalent to maximizing \(\sum_i n_i^2\).

In any maximizer of \(\sum_i n_i^2\) under these constraints, at most one part size is strictly smaller than \(n-k\).
Indeed, suppose there are two indices \(i\neq j\) with \(1\le n_i\le n_j<n-k\). Replacing \((n_i,n_j)\mapsto(n_i-1,n_j+1)\), where a zero part (if \(n_i=1\)) is discarded, keeps feasibility and strictly increases \(\sum_i n_i^2\), since \((n_i-1)^2+(n_j+1)^2-(n_i^2+n_j^2)=2(n_j-n_i)+2>0\).
This contradicts maximality.
Therefore all but at most one parts must have size exactly \(n-k\), giving the stated partition.
Writing \(n=q(n-k)+b\), \(0\le b<n-k\), yields that form, and substituting into \eqref{eq:cmp_edge} gives \eqref{eq:env_mk}.
\hfill$\qed$

Proposition \ref{prop:env_formula} yields a closed-form staircase of strict-maximal points in \(\mathcal K_n\), providing a cheap dense-regime backbone library.
Detailed envelope-versus-baseline performance comparisons are deferred to Figures \ref{fig:backbone_only_lambda2}--\ref{fig:backbone_only_lambda2_low}.

At very sparse densities, the first two envelope levels are key anchors: at \(k=1\), Proposition~\ref{prop:env_formula} gives \(K_{1,n-1}\), i.e., the star graph, which is optimal at the minimum connected budget \(m=n-1\) \cite{Mohar91}; at \(k=2\), it gives \(K_{2,n-2}\), whose optimality at the corresponding budget is a well-known open conjecture \cite{Kolokolnikov15}.

For backbone construction, given target \((n,m)\), choose \(k^\star=\max\{k\in\{1,\dots,n-2\}:m_{\mathrm{env}}(k)\le m\}\); because \(m_{\mathrm{env}}(k)\) is monotone nondecreasing, \(k^\star\) is found by integer binary search in \(O(\log n)\) time.
The envelope candidate is then the complete multipartite graph from Proposition \ref{prop:env_formula} at level \(k^\star\) (equivalently, using \eqref{eq:env_division}--\eqref{eq:env_mk}).
Its algebraic-connectivity score is \(\lambda_2=k^\star\) (or \(\lambda_2=n\) when \(m=\binom{n}{2}\)), and this candidate is passed to the routing/selection stage in Section II-C.
Algorithm~\ref{alg:env_backbone} summarizes this deterministic envelope construction.

\begin{algorithm}[th]
\caption{Envelope Backbone Construction}
\label{alg:env_backbone}
\begin{algorithmic}[1]
\State \textbf{Input:} \(n,m\)
\State \textbf{Output:} Envelope backbone \(G_{\mathrm{env}}\) with \(|E(G_{\mathrm{env}})|\le m\)
\If{\(m=\binom{n}{2}\)}
\State return \(K_n\) with \(\lambda_2=n\)
\EndIf
\State \(k^\star\leftarrow\max\{k\in\{1,\dots,n-2\}:m_{\mathrm{env}}(k)\le m\}\) by integer binary search on \eqref{eq:env_mk}
\State \(q\leftarrow \left\lfloor n/(n-k^\star)\right\rfloor\), \(b\leftarrow n-q(n-k^\star)\)
\State build \(K_{n_1,\dots,n_r}\) with \(q\) parts of size \(n-k^\star\) and one part \(b\) if \(b>0\), \(G_{\mathrm{env}}\leftarrow K_{n_1,\dots,n_r}\)
\State return \(G_{\mathrm{env}}\) and \(\lambda_2(G_{\mathrm{env}})=k^\star\)
\end{algorithmic}
\end{algorithm}

\begin{figure}[t]
    \centering
    \includegraphics[width=0.99\linewidth]{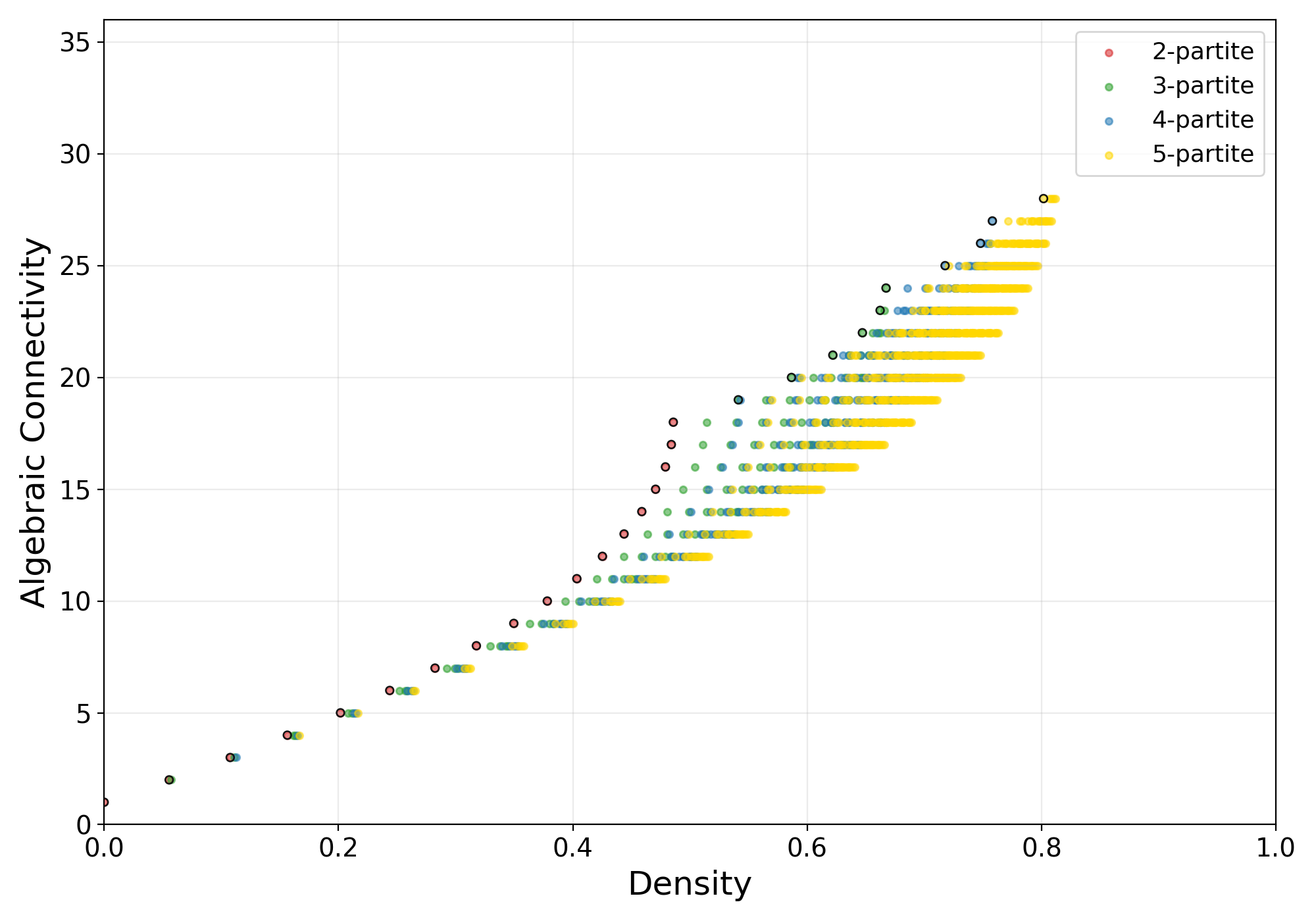}
    \vspace{-.15in}
    \caption{Complete 2/3/4/5-partite graphs in \(\mathcal K_{36}\) on the \((\rho,\lambda_2)\) plane.}
    \label{fig:envelope_staircase}
\end{figure}

\subsection{Cayley Backbone via Coset-Based Generating Sets}
While envelope points from \(\mathcal K_n\) are strong in dense regimes, our experiments show a clear weakness for \(\rho<\tfrac12\), where their \(\lambda_2\) is non-competitive (see Figure \ref{fig:backbone_only_lambda2_low}).
Therefore, the need for a second backbone family is driven not only by limited envelope coverage across feasible density levels, but also by quality degradation in low-to-intermediate densities.
Balanced complete multipartite graphs are a natural anchor because they are guaranteed optimal at their corresponding budgets \cite{ishii2016extensions}.
This raises a natural question: beyond \(\mathcal K_n\), is there a structured graph family that extends balanced complete multipartite graphs, yet can realize stronger candidates in other density ranges?
Cayley graphs provide such a generalization; as shown below, complete balanced multipartite graphs arise as a special Cayley case.

We next give a formal definition of Cayley graphs.

\begin{definition}
Let \(\Gamma\) be a finite group and \(S\subseteq \Gamma\setminus\{e\}\) with \(S=S^{-1}\).
The undirected Cayley graph \(\mathrm{Cay}(\Gamma,S)\) has vertex set \(\Gamma\) and edge set \(\{\{g,gs\}: g\in\Gamma,\ s\in S\}\).
\end{definition}

The inverse-closed condition ensures undirected edges, while group actions provide strong symmetry and reusable structure for sampling.

The next proposition establishes that the balanced complete multipartite case is a special instance of Cayley graphs.

\begin{proposition}[Balanced Multipartite as Cayley]
\label{prop:balanced_cmp_cayley}
Let \(\Gamma\) be a finite group, and let \(H<\Gamma\) be a proper subgroup of index \(r\ge 2\).
Define \(S=\Gamma\setminus H\).
Then \(\mathrm{Cay}(\Gamma,S)\) is a complete balanced \(r\)-partite graph whose parts are the left cosets of \(H\).
\end{proposition}

{\em Proof:}
Since \(e\in H\), we have \(e\notin S\). Also \(S^{-1}=S\), so \(\mathrm{Cay}(\Gamma,S)\) is an undirected Cayley graph.
Let \(\Gamma=\bigsqcup_{\ell=1}^r g_\ell H\) be the left-coset decomposition, where \(\bigsqcup\) denotes disjoint union. Each part has size \(|H|\), so the partition is balanced.
For \(x,y\in\Gamma\), if they are in the same coset then \(x^{-1}y\in H\), hence \(x^{-1}y\notin S\), so no edge.
If they are in different cosets then \(x^{-1}y\notin H\), hence \(x^{-1}y\in S\), so an edge exists.
Therefore adjacency is exactly across distinct cosets, i.e., a complete balanced \(r\)-partite graph.
\hfill$\qed$

Proposition~\ref{prop:balanced_cmp_cayley} gives the key motivation for our construction: index-\(r\) coset-based generators recover the balanced complete \(r\)-partite when all nontrivial cosets are selected, and random inverse-closed sampling from these cosets extends that structure to lower density levels.
Numerically, we observe that generic random Cayley candidates are already competitive, while coset-based Cayley candidates are typically stronger.

An index-\(r\) subgroup requires that \(r\) divides \(n\), so direct coset construction is not always available.
To handle arbitrary \(n\), we use a lift-and-delete extension: construct a Cayley graph on \(n+j\) vertices where divisibility holds, then delete \(j\) vertices to return to size \(n\).
Numerically, lift-and-delete remains competitive for small lift sizes (e.g., \(j=1,2\)).
The following lemma indicates that deleting one vertex is equivalent to deleting any other vertex up to graph isomorphism, which simplifies decision-making in the \(j=1\) case.

\begin{lemma}[Vertex-Transitivity of Cayley Graphs]
\label{lem:cayley_vertex_transitive}
Let \(\Gamma\) be a finite group and \(S\subseteq \Gamma\setminus\{e\}\) with \(S=S^{-1}\).
Then \(\mathrm{Cay}(\Gamma,S)\) is vertex-transitive.
\end{lemma}

{\em Proof:}
This is a standard property of Cayley graphs; see Theorem 3.1.2 in \cite{godsil_royle_agt_2001}.
\hfill$\qed$

For \(j=1\), deletion is choice-free; for \(j>1\), deletion is not choice-free, so we use random deletion and leave principled selection for future work.

How the choice of group family and subgroup affects algebraic connectivity has not been systematically investigated and is left as a future direction.
For implementation simplicity, we use standard group families (semidihedral, dihedral, and cyclic) and fixed subgroup templates.
We therefore define the Cayley backbone as a coset-based generating-set Cayley construction, with lift-and-delete applied when divisibility constraints prevent direct index-\(r\) construction.
Algorithm~\ref{alg:cayley_backbone} gives the corresponding fixed-index-\(r\) backbone procedure.
When the feasible degree is too small (e.g., \(k=1\) or \(k=2\)), sampled Cayley candidates can be disconnected.

\begin{algorithm}[th]
\caption{Index-\(r\) Cayley Backbone Construction}
\label{alg:cayley_backbone}
\begin{algorithmic}[1]
\State \textbf{Input:} \(n,m,r\), sample count \(t\ge 1\)
\State \textbf{Output:} Cayley backbone \(G_{\mathrm{cay}}\)
\State compute lift-delete \(j\), set \(n_{\mathrm{lift}}\leftarrow n+j\)
\State set \(k\leftarrow\min\{n_{\mathrm{lift}}-1,\lfloor 2m/(n-j)\rfloor\}\)
\State choose preset group \(\Gamma\) and index-\(r\) subgroup \(H\) (with lift if needed)
\State set \(k\) to the largest degree \(\le k\) feasible for the coset sampler
\State set \(\widehat{\tau}\leftarrow \frac{n_{\mathrm{lift}}^2}{2}\cdot\frac{r-1}{r}\)
\If{\(j=0\) and \(m=\widehat{\tau}\)}
\State return balanced complete \(r\)-partite graph
\EndIf
\State initialize \(\mathcal{C}\leftarrow\emptyset\); sample one shared deletion set \(D\), \(|D|=j\)
\If{\(k\le |\Gamma\setminus H|\)}
\State sample \(t\) inverse-closed generating sets \(S_1,\ldots,S_t\) of size \(k\) from \(\Gamma\setminus H\)
\Else
\State sample \(t\) inverse-closed generating sets by taking all elements in \(\Gamma\setminus H\) and filling the remaining \(k-\lvert\Gamma\setminus H\rvert\) generators from \(H\setminus\{e\}\)
\EndIf
\State for each distinct sampled \(S_s\), form \(\widetilde G_s=\mathrm{Cay}(\Gamma,S_s)-D\); add it to \(\mathcal C\) only if connected and \(|E(\widetilde G_s)|\le m\)
\State if \(\mathcal C=\emptyset\), return failure
\State return \(G_{\mathrm{cay}}=\arg\max_{G\in\mathcal{C}}\lambda_2(G)\)
\end{algorithmic}
\end{algorithm}

At the coverage level, index-\(r\) Cayley construction reaches edge ratios up to approximately \((r-1)/r\). Hence indices \(2,3,4\) cover low-to-intermediate densities up to approximately \(3/4\), i.e., most regimes where the envelope points are not dominant in practice.
Based on this coverage property, we now instantiate a deterministic density-routing rule that decides which family (or both) is evaluated at each target density, and how Cayley indices are selected under the corresponding edge feasibility.

\subsection{Density Routing and Backbone Selection Rule}
We operationalize the above family-level observations using a compact overlap rule:
envelope is used for \(\rho\in[0,0.1]\cup[0.5,1]\), Cayley is used for \(\rho\in[0.05,0.75]\), and both families are evaluated on the overlap windows \([0.05,0.1]\) and \([0.5,0.75]\).
For Cayley candidates, we use index-\(2/3/4\) coset constructions with inverse-closed sampling under edge-budget feasibility.
Specifically, index-2 uses semidihedral templates for power-of-two \(n\ge16\) without lifting and dihedral templates otherwise, while indices 3 and 4 use cyclic templates.
For each index \(r\in\{2,3,4\}\), let \(j_r\) be the lift-delete count.
For routing, \(k_{\max}^{(r)}\) is an optimistic degree cap, \(m_{\mathrm{post}}^{(r)}\) is a projected edge count, and \(d_{\mathrm{post}}^{(r)}\) is a density-based pocket coordinate:
\begin{equation}
\begin{aligned}
k_{\max}^{(r)}&=\min\!\left\{n+j_r-1,\bigg\lfloor\frac{2m}{n-j_r}\bigg\rfloor\right\},\\
m_{\mathrm{post}}^{(r)}&=\bigg\lfloor\frac{k_{\max}^{(r)}(n-j_r)}{2}\bigg\rfloor,\quad
d_{\mathrm{post}}^{(r)}=\frac{2m_{\mathrm{post}}^{(r)}}{n^2}.
\end{aligned}
\label{eq:cayley_pocket}
\end{equation}
Deleting \(D\), \(|D|=j_r\), from a degree-\(k\) graph gives
\[
|E(G-D)|=k(n-j_r)/2+|E(G[D])|.
\]
We therefore check actual edge counts and connectivity before accepting candidates.

Without deletion, at balanced index-\(r\) endpoints, \(d_{\mathrm{post}}^{(r)}\) equals \(1-\frac{1}{r}\), i.e., \(1/2,2/3,3/4\) for \(r=2,3,4\).
Density-based pocket tests are then applied as \(d_{\mathrm{post}}^{(2)}\in(0,0.6)\), \(d_{\mathrm{post}}^{(3)}\in(\tfrac{1}{2},\tfrac{2}{3})\), and \(d_{\mathrm{post}}^{(4)}\in(\tfrac{2}{3},\tfrac{3}{4})\), evaluating all valid indices and retaining the feasible Cayley candidate with largest initial \(\lambda_2\).
This matches the practical coverage of index-\(2/3/4\) Cayley constructions at approximately \(1/2,2/3,3/4\), respectively (exact in no-lift cases). Empirically, for \(\rho>3/4\), envelope points are typically stronger (see Figures \ref{fig:backbone_only_lambda2} and \ref{fig:backbone_only_lambda2_dense}), so routing favors envelope there.
In overlap windows, both families are evaluated; each attempts to produce a backbone, and when both are feasible, the final family choice is made after completion by comparing final \(\lambda_2\).
In Cayley-only windows, if sampled Cayley candidates are infeasible/disconnected (typically at very small feasible degree), envelope is used as fallback.

The routing rule specifies \emph{when} each family is activated; Algorithm~\ref{alg:backbone} summarizes \emph{how} the backbone stage combines envelope construction (Algorithm~\ref{alg:env_backbone}), Cayley construction (Algorithm~\ref{alg:cayley_backbone}), overlap handling, and fallback.

\subsection{Backbone Complexity and Scalability}
For fixed \((n,m)\), envelope-level lookup is \(O(\log n)\) via monotone integer binary search. Adjacency construction costs \(O(n^2)\). Intact Cayley spectra use group representations \cite{babai79spectra}: per candidate, cyclic FFTs and semidihedral \(2\times2\) Hermitian Fourier blocks cost \(O(n\log n)\); direct dihedral summation costs \(O(nk)\), \(k=|S|\). After deletion, a dense selected-eigenvalue solver costs \(O(n^3)\) in the worst case. Candidate evaluations are independent and parallelizable, but run serially here.

Because backbone budgets are \(O(n)\)-level discretizations of a \(\Theta(n^2)\) feasible edge range, the residual gap to target is typically \(\Theta(n)\) edges. Hence, when completion is the dominant sequential stage, backbone initialization can reduce the effective horizon by about one order (\(\Theta(n^2)\to\Theta(n)\)), consistent with observed scaling.

In practice, we use fixed \(t=500\) sampling attempts per eligible index, removing duplicate generating sets. Increasing \(t\) from 500 up to \(10^4\) gives marginal quality gains in tested ranges. Although required \(t\) may increase with \(n\) in principle, fixed-\(t\) quality remains stable in our tested ranges.

\begin{algorithm}[th]
\caption{Density-Dependent Backbone Construction}
\label{alg:backbone}
\begin{algorithmic}[1]
\State \textbf{Input:} \(n,m\), Cayley sampling count \(t\)
\State \textbf{Output:} Backbone candidate list \(\mathcal B\)
\State compute \(\rho\) from \((n,m)\) via \eqref{eq:rho_connected}
\State set \(\texttt{env\_active}\leftarrow[\rho\in[0,0.1]\cup[0.5,1]]\), \(\texttt{cay\_active}\leftarrow[\rho\in[0.05,0.75]]\)
\State initialize \(\mathcal B\leftarrow[\,]\)
\If{\(\texttt{env\_active}\)}
\State \(G_{\mathrm{env}}\leftarrow\) Algorithm~\ref{alg:env_backbone}\((n,m)\)
\State append \(G_{\mathrm{env}}\) to \(\mathcal B\)
\EndIf
\If{\(\texttt{cay\_active}\)}
\State run Algorithm~\ref{alg:cayley_backbone}\((n,m,r,t)\) for every eligible index \(r\) from \eqref{eq:cayley_pocket}
\If{a connected Cayley candidate exists}
\State append the best feasible Cayley candidate to \(\mathcal B\)
\EndIf
\EndIf
\If{\(\mathcal B\) is empty}
\State \(G_{\mathrm{env}}\leftarrow\) Algorithm~\ref{alg:env_backbone}\((n,m)\) \Comment{fallback}
\State append \(G_{\mathrm{env}}\) to \(\mathcal B\)
\EndIf
\State return \(\mathcal B\)
\end{algorithmic}
\end{algorithm}

\section{Reinforcement Learning Completion} \label{sec:RL}

After backbone construction, we solve only the residual edge-allocation problem by a sequential policy.
Given an initial connected backbone \(G_0=(V,E_0)\) with \(|E_0|\le m_{\mathrm{target}}\), each step adds one non-edge for \(T=m_{\mathrm{target}}-|E_0|\) steps.
Thus the RL stage does not design a graph from scratch; it refines a density-guided backbone under a fixed residual budget.

\subsection{Residual Edge-Addition MDP}
A state is the current graph \(G_t=(V,E_t)\), and an action is a non-edge \((i,j)\notin E_t\).
The transition is the update \(E_{t+1}=E_t\cup\{(i,j)\}\), and the episode terminates at \(t=T\) (equivalently \(|E_t|=m_{\mathrm{target}}\)).
In our implementation, episodes are sampled over graph sizes \(12\le n\le24\) and random target densities, then mapped to connected-range edge budgets by \eqref{eq:rho_connected}.

\subsection{State and Candidate Representation}
The policy uses structural node, pair, and global features:

\[
\begin{aligned}
x_i&=\big[d_i/(n-1),\ d_i^{(2)}/(n-1),\ C_i\big],\\
p_{ij}&=\big[\mathrm{CN}^{\mathrm{norm}},\ \mathrm{Jaccard},\ \delta_{ij},\ u_{ij},\ v_{ij}\big],\\
g_t&=\big[\log(n)/\log(64),\ \rho_{\mathrm{target}},\ \rho_t\big].
\end{aligned}
\]

Here \(d_i\) is the node degree, \(d_i^{(2)}\) is the corresponding two-hop neighbor count (excluding the node itself and one-hop neighbors), \(C_i\) is the local clustering coefficient, 
\begin{align*}
\mathrm{CN}^{\mathrm{norm}} = \frac{|N(i)\cap N(j)|}{n-2}, \;\;\;
\mathrm{Jaccard} = \frac{|N(i)\cap N(j)|}{|N(i)\cup N(j)|},
\end{align*} 
\(\mathrm{dist}(i,j)\) is shortest-path distance in \(G_t\), \(\delta_{ij}=\min\{\mathrm{dist}(i,j),4\}/4\), \(u_{ij}=(d_i+d_j)/(2(n-1))\), and \(v_{ij}=|d_i-d_j|/(n-1)\).

We form nonedges among up to 32 nodes ranked by a weighted sum of their features, supplement to 32 pairs when available, and retain at most 96 using degree and distance scores. Inference requires no intermediate spectral computations.

\subsection{Policy and Value Networks}
For each candidate non-edge \((i,j)\), we form the symmetric representation
\[
z_{ij}=[\,|x_i-x_j|,\ x_i\odot x_j,\ p_{ij},\ g_t\,]\in\mathbb R^{14}.
\]
It is scored by a multilayer perceptron (MLP) with two 64-unit ReLU hidden layers and one scalar output.

The action distribution is the softmax over candidate scores.
Evaluation selects the largest score. A separate value MLP with the same hidden widths uses \([n^{-1}\sum_i x_i,g_t]\).

\subsection{Training Objective}
Completion is trained with proximal policy optimization (PPO) \cite{schulman2017ppo}, using a value baseline and entropy regularization.
The per-step reward is
\[
r_t=\frac{\lambda_2(G_{t+1})-\lambda_2(G_t)}{n}+\beta\cdot \mathbf{1}_{\mathrm{terminal}}\frac{\lambda_2(G_T)}{n},
\]
where \(\beta\) is the terminal-bonus coefficient, set to 0.2.

We use generalized advantage estimates \(\widehat A_t\) with \(\gamma=0.99\), \(\lambda_{\mathrm{GAE}}=0.95\), value targets \(\mathcal R_t=\widehat A_t+V_{\mathrm{old}}(s_t)\), and rollout-standardized policy advantages \(A_t\).

With importance ratio \(\eta_t(\theta)=\pi_\theta(a_t|s_t)/\pi_{\theta_{\mathrm{old}}}(a_t|s_t)\) and clipped ratio
\(\bar{\eta}_t(\theta)=\mathrm{clip}(\eta_t(\theta),1-\epsilon,1+\epsilon)\), we use:
\[
\begin{aligned}
\mathcal L_{\mathrm{PPO}}
=&-\mathbb E\!\left[\min\!\left(\eta_t(\theta)A_t,\ \bar{\eta}_t(\theta)A_t\right)\right]\\
&+c_v\,\mathbb E\!\left[(V_\psi(s_t)-\mathcal{R}_t)^2\right]
-c_e\,\mathbb E\!\left[\mathcal H(\pi_\theta(\cdot|s_t))\right].
\end{aligned}
\]
We evaluate a checkpoint after \(3\times10^4\) training transitions, using \(\epsilon=0.2\), \(c_v=0.5\), \(c_e=0.01\), and initial Adam learning rate \(3\times10^{-4}\).

\section{Numerical Results} \label{sec:simulation}

This section evaluates the proposed framework in terms of solution quality, computational efficiency, and scalability across different graph sizes and density regimes.

\begin{table*}[t]
    \centering
    \caption{Quality summary. FVG: absolute \(\lambda_2\) (mean\(\pm\)std). Others: \(\Delta\lambda_2\) mean\(\pm\)std / win\% / match\%. \textcolor{red}{Red} = best \(\Delta\lambda_2\); \textcolor{blue}{blue} = best win\%.}
    \label{tab:baseline_lambda2_summary}
    \footnotesize
    \setlength{\tabcolsep}{1.9pt}
    \renewcommand{\arraystretch}{1.00}
    \resizebox{0.96\textwidth}{!}{%
    \begin{tabular}{lccccc}
        \toprule
        Method & \(n=8\) & \(n=16\) & \(n=32\) & \(n=64\) & \(n=128\) \\
        \midrule
        FVG      & \(3.34\pm1.77\) & \(6.72\pm4.18\) & \(13.99\pm8.81\) & \(28.87\pm18.19\) & \(59.16\pm36.84\) \\
        \midrule
        MAC      & \(-0.47\pm0.41\)/0/25  & \(-1.27\pm0.65\)/1/2  & \(-3.03\pm1.42\)/0/2  & \(-6.62\pm3.10\)/0/1  & \(-13.72\pm6.52\)/0/1 \\
        MDMD     & \(-0.08\pm0.35\)/10/50 & \(-0.47\pm0.40\)/7/12 & \(-1.09\pm0.73\)/4/5  & \(-2.15\pm1.49\)/1/3  & \(-3.89\pm2.49\)/1/1 \\
        WTS      & \(0.27\pm0.36\)/50/50  & \(0.18\pm0.22\)/\bestwin{83}/12 & \(-0.23\pm0.31\)/12/2 & \(-0.76\pm0.63\)/2/2  & \(-1.63\pm1.12\)/1/1 \\
        k-opt    & \(0.19\pm0.29\)/45/55  & \(0.20\pm0.16\)/81/19 & \(0.13\pm0.13\)/\bestwin{90}/10 & \(0.04\pm0.05\)/\bestwin{80}/20 & \(0.01\pm0.02\)/58/8 \\
        SDP-step & \(0.08\pm0.57\)/30/40  & \(0.29\pm0.57\)/67/11 & \(0.39\pm0.71\)/71/4  & --                     & -- \\
        OA-PMC   & \bestdelta{\(0.31\pm0.40\)}/\bestwin{55}/45  & --                     & --                     & --                     & -- \\
        BB-RL(Ours) & \(0.25\pm0.42\)/40/50 & \bestdelta{\(0.45\pm0.62\)}/66/11 & \bestdelta{\(0.73\pm0.85\)}/71/3 & \bestdelta{\(1.13\pm1.22\)}/73.2/2 & \bestdelta{\(1.61\pm2.00\)}/\bestwin{66.6}/1.6 \\
        \bottomrule
    \end{tabular}
    }
\end{table*}

\begin{table}[t]
    \centering
    \caption{Mean online runtime (s). \textcolor{red}{Red}: faster of matched MAC/BB-RL runs. \(\dagger\): historical timings. BB-RL completion includes final selection.}
    \label{tab:baseline_runtime_summary}
    \footnotesize
    \setlength{\tabcolsep}{1.3pt}
    \renewcommand{\arraystretch}{1.05}
    \resizebox{\columnwidth}{!}{%
    \begin{tabular}{lccccc}
        \toprule
        Method & \(n=8\) & \(n=16\) & \(n=32\) & \(n=64\) & \(n=128\) \\
        \midrule
        MAC      & \besttime{\(4.35\times10^{-3}\)} & \besttime{\(6.31\times10^{-3}\)} & \(1.62\times10^{-2}\) & \(4.63\times10^{-2}\) & \(1.53\times10^{-1}\) \\
        MDMD\(^{\dagger}\) & \(8.59\times10^{-4}\) & \(5.42\times10^{-3}\) & \(4.92\times10^{-2}\) & \(4.08\times10^{-1}\) & \(5.14\) \\
        FVG\(^{\dagger}\) & \(1.38\times10^{-3}\) & \(7.46\times10^{-3}\) & \(1.02\times10^{-1}\) & \(9.27\times10^{-1}\) & \(24.06\) \\
        WTS\(^{\dagger}\) & \(5.47\times10^{-1}\) & \(1.00\) & \(2.14\) & \(5.53\) & \(19.78\) \\
        k-opt\(^{\dagger}\) & \(5.20\times10^{-2}\) & \(4.04\times10^{-1}\) & \(2.70\) & \(8.60\) & \(38.78\) \\
        SDP-step\(^{\dagger}\) & \(1.48\times10^{-1}\) & \(1.42\) & \(24.88\) & -- & -- \\
        OA-PMC\(^{\dagger}\) & \(10.62\) & -- & -- & -- & -- \\
        \midrule
        BB-RL(Ours, total) & \(4.85\times10^{-3}\) & \(7.76\times10^{-3}\) & \besttime{\(1.50\times10^{-2}\)} & \besttime{\(3.84\times10^{-2}\)} & \besttime{\(1.20\times10^{-1}\)} \\
        \hspace{0.6em}init & \(2.25\times10^{-3}\) & \(2.76\times10^{-3}\) & \(4.82\times10^{-3}\) & \(1.53\times10^{-2}\) & \(6.08\times10^{-2}\) \\
        \hspace{0.6em}completion & \(2.60\times10^{-3}\) & \(5.00\times10^{-3}\) & \(1.02\times10^{-2}\) & \(2.31\times10^{-2}\) & \(5.90\times10^{-2}\) \\
        \bottomrule
    \end{tabular}
    }
\end{table}

\subsection{Experimental Setup}
All methods are evaluated under the same edge-budget mapping \eqref{eq:rho_connected}.
We use 20 nontrivial edge budgets at \(n=8\) and 100 at each \(n\in\{16,32,64,128\}\), excluding the path and complete-graph endpoints. Each point maps to the same \((n,m)\) target across methods.
Each \((n,\rho)\) point is repeated 5 times; tables report mean/std over the density sweep from these per-point repeats. The quality metric is \(\lambda_2\) at the target budget, and runtime is wall-clock time.
Experiments are conducted on the Seawulf HPC cluster, long-40core partition, in exclusive-node mode using 1 node (8 CPU cores, 64\,GB RAM) per run.
In all experiments, sample-wise Cayley parallelization acceleration is disabled.
Most reported \(n\) are powers of two for clarity of scaling, but additional odd/even/prime/random \(n\) show similar qualitative behavior in our experiments; a finer characterization of how performance varies with integer structure of \(n\) is left for future work.
Implementation details, hyperparameter settings, and experiment scripts are available in our public repository~\cite{zhu2026code}.

\subsection{Baseline Comparison Among Existing Methods}\label{subsec:benchmark}
This subsection benchmarks full-construction baselines.
Compared methods include constructive heuristics (FVG \cite{ghosh2006growing}, MDMD \cite{8412478}), local search (WTS \cite{wei2014algebraic}, k-opt \cite{somisetty2024spectral}), relaxation/optimization (MAC \cite{9981584}, SDP-step \cite{wei2014algebraic}), and exact OA-PMC \cite{somisetty2025optimal}.
Tables \ref{tab:baseline_lambda2_summary}--\ref{tab:baseline_runtime_summary} summarize quality/runtime, and Figures \ref{fig:baseline_lambda2}--\ref{fig:baseline_lambda2_small} show representative quality curves.

\begin{figure}[ht]
    \centering
    \includegraphics[width=0.99\linewidth]{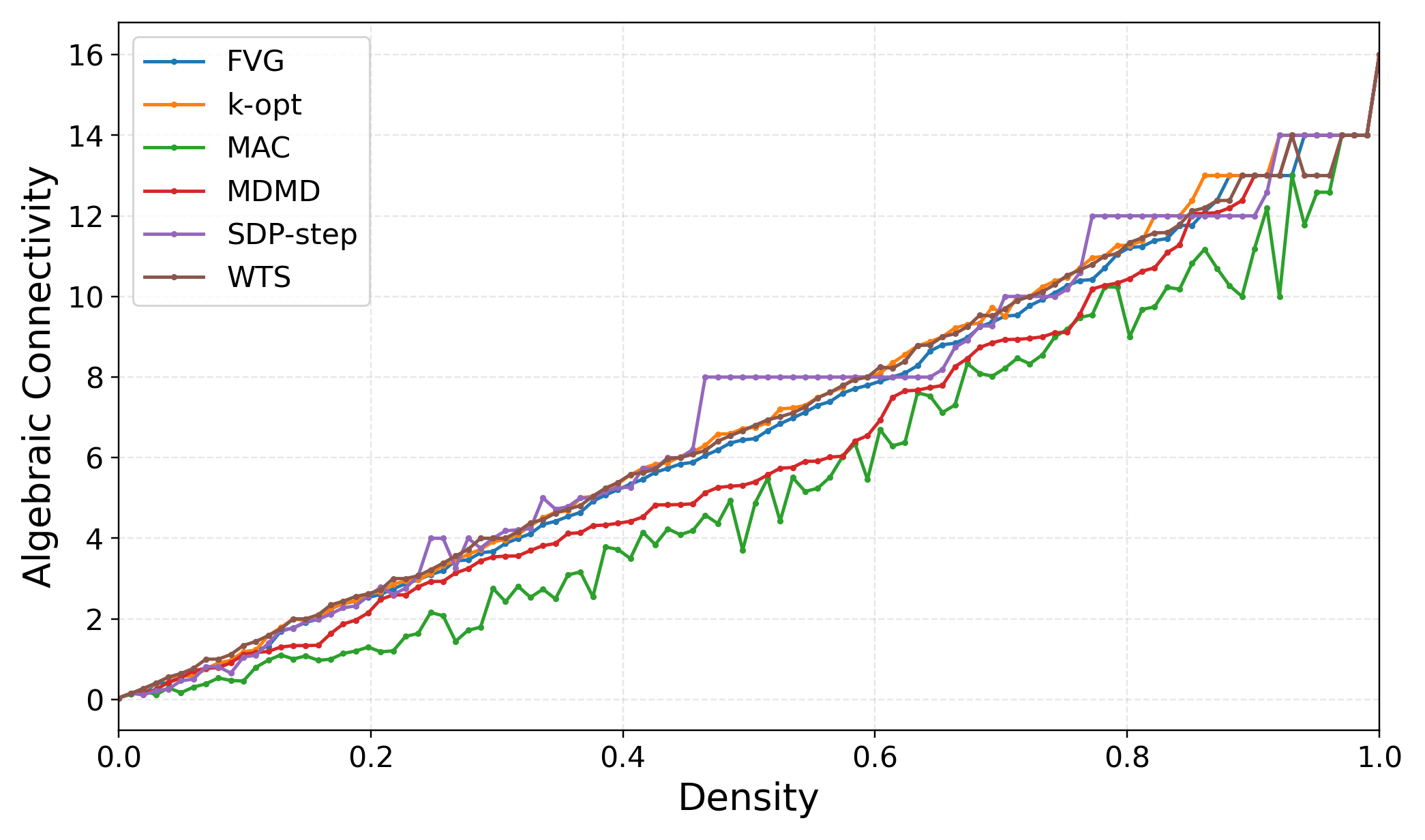}
\vspace{-.2in}    
    \caption{Baseline quality at \(n=16\).}
    \label{fig:baseline_lambda2}
\end{figure}

\begin{figure}[ht]
    \centering
    \includegraphics[width=0.95\linewidth]{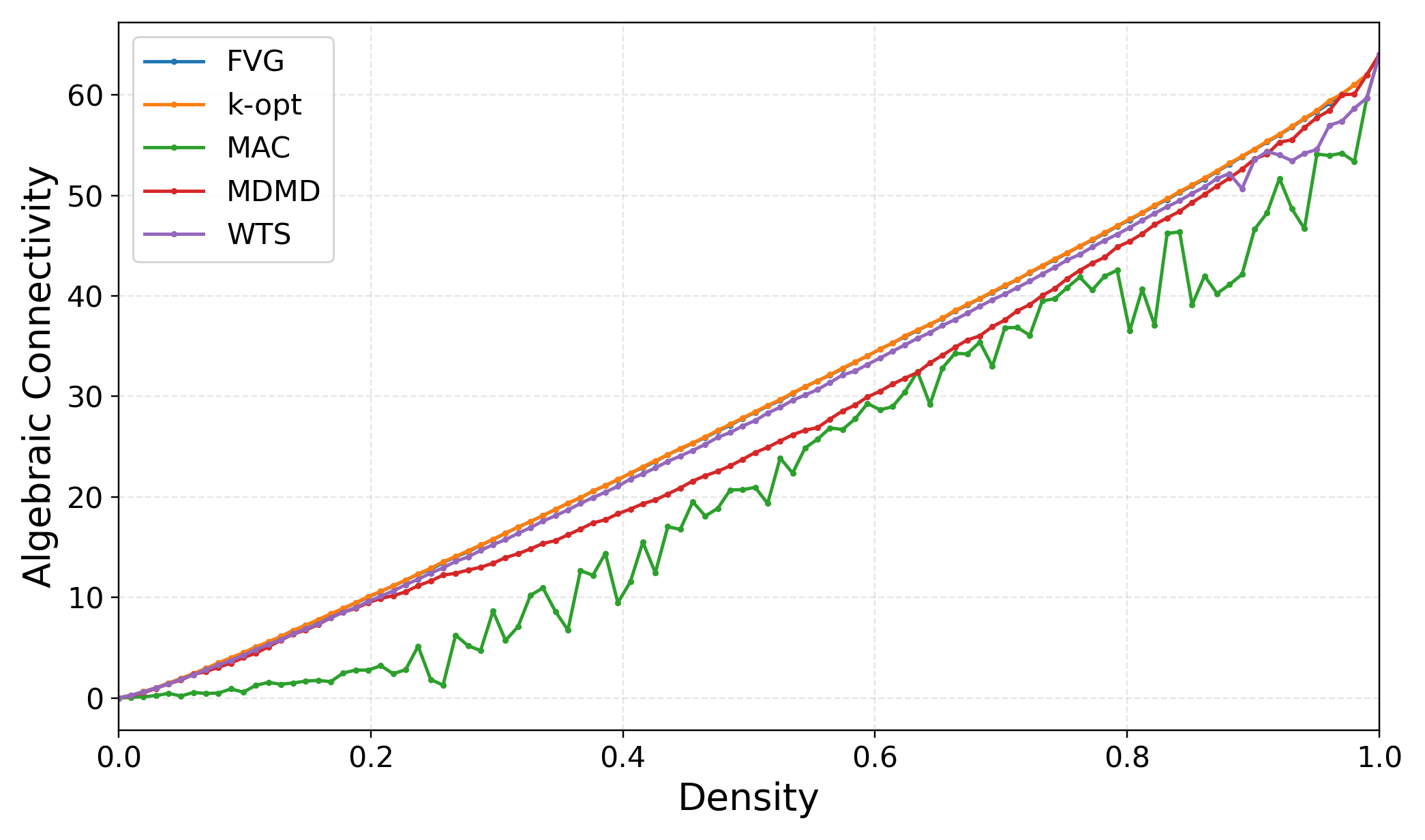}
    \caption{Baseline quality at \(n=64\).}
    \label{fig:baseline_lambda2_small}
\end{figure}

Results show the expected trade-off: MAC/MDMD are fast but lower in $\lambda_2$; WTS/k-opt are usually slower and only occasionally better at small \(n\); SDP-step/OA-PMC are higher in $\lambda_2$ when applicable but much more expensive. Hence FVG is the primary practical full-construction baseline.

\subsection{Backbone-Only Performance Across Density Regimes}\label{subsec:backbone_only}
We compare backbone-only outputs (no RL completion), including EIG \cite{Shahbaz23}, a complete-multipartite constructor based on complement-component decomposition that may return \(m\ge m_{\text{target}}\).
Figures \ref{fig:backbone_only_lambda2}--\ref{fig:backbone_only_lambda2_low} show dense and low-to-intermediate regimes; the envelope optimality-certified subset follows Corollary 4.6 of \cite{Shahbaz23}.
Overall, envelope dominates dense regimes, while Cayley is competitive in low-to-intermediate regimes.

\begin{figure}[t]
    \centering
    \includegraphics[width=0.95\linewidth]{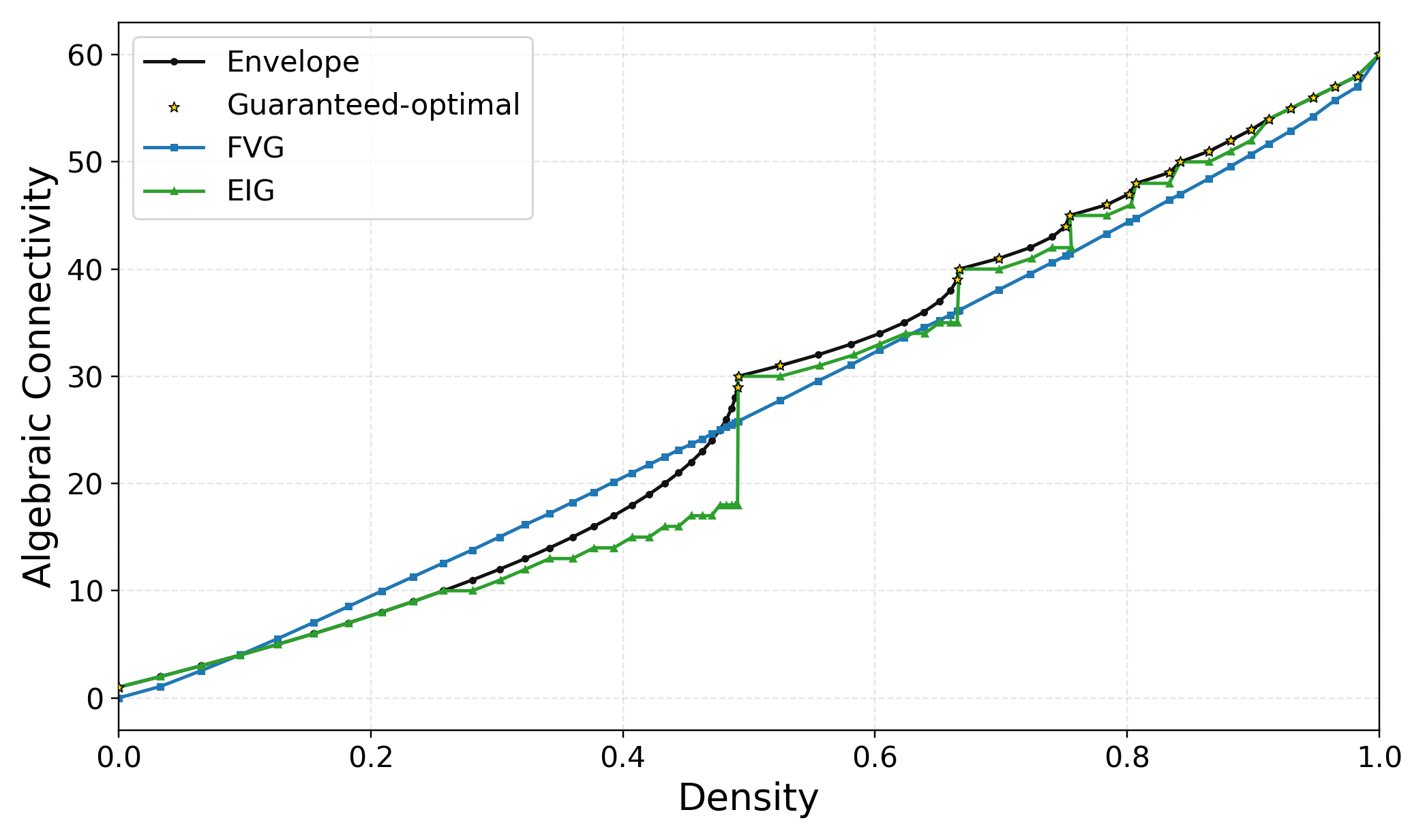}
    \caption{Envelope, FVG, and EIG at \(n=60\).}
    \label{fig:backbone_only_lambda2}
\end{figure}

\begin{figure}[t]
    \centering
    \includegraphics[width=0.95\linewidth]{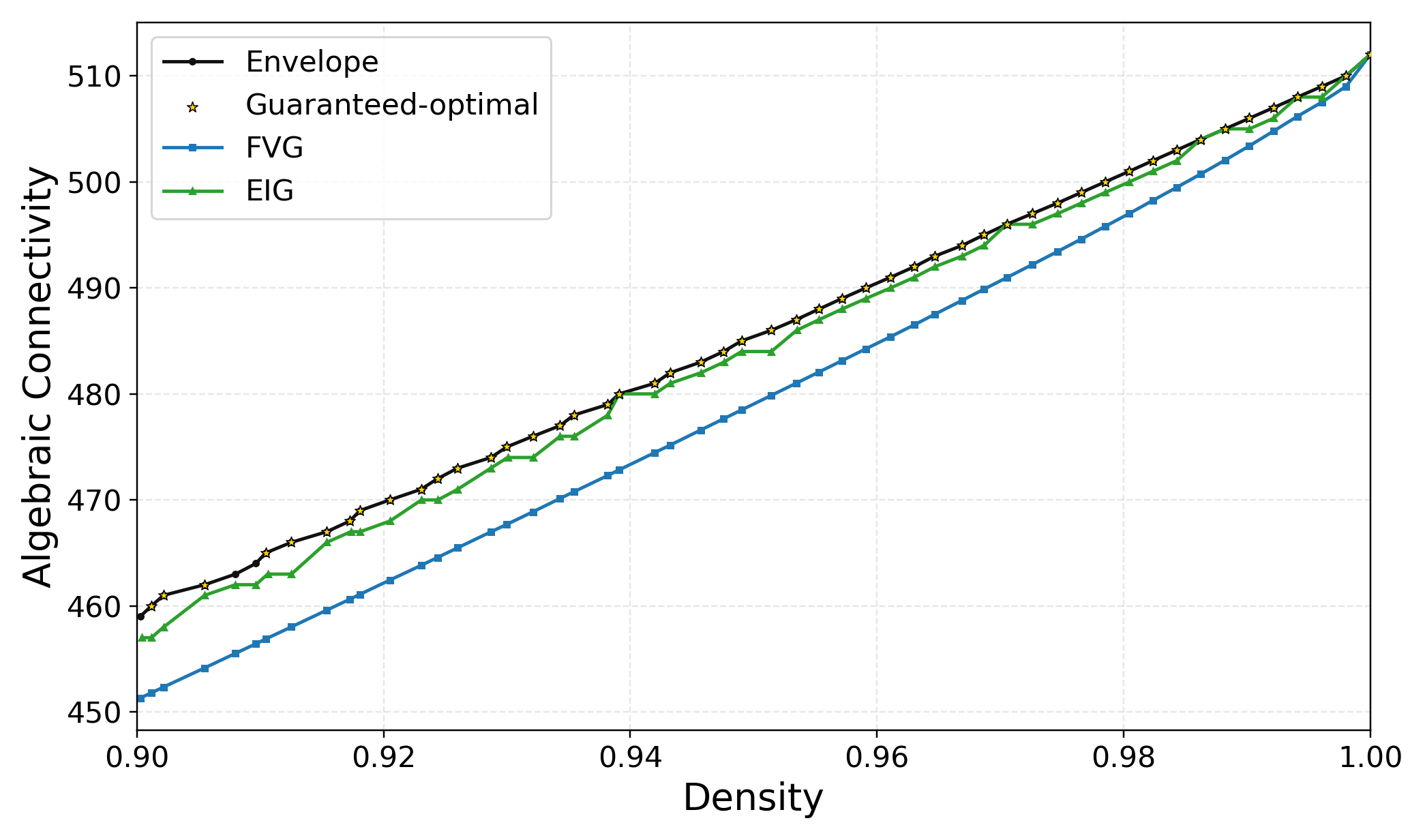}
    \caption{Envelope, FVG, and EIG at \(n=512\), \(\rho>0.9\).}
    \label{fig:backbone_only_lambda2_dense}
\vspace{-.1in}
\end{figure}

\begin{figure}[t]
    \centering
    \includegraphics[width=0.95\linewidth]{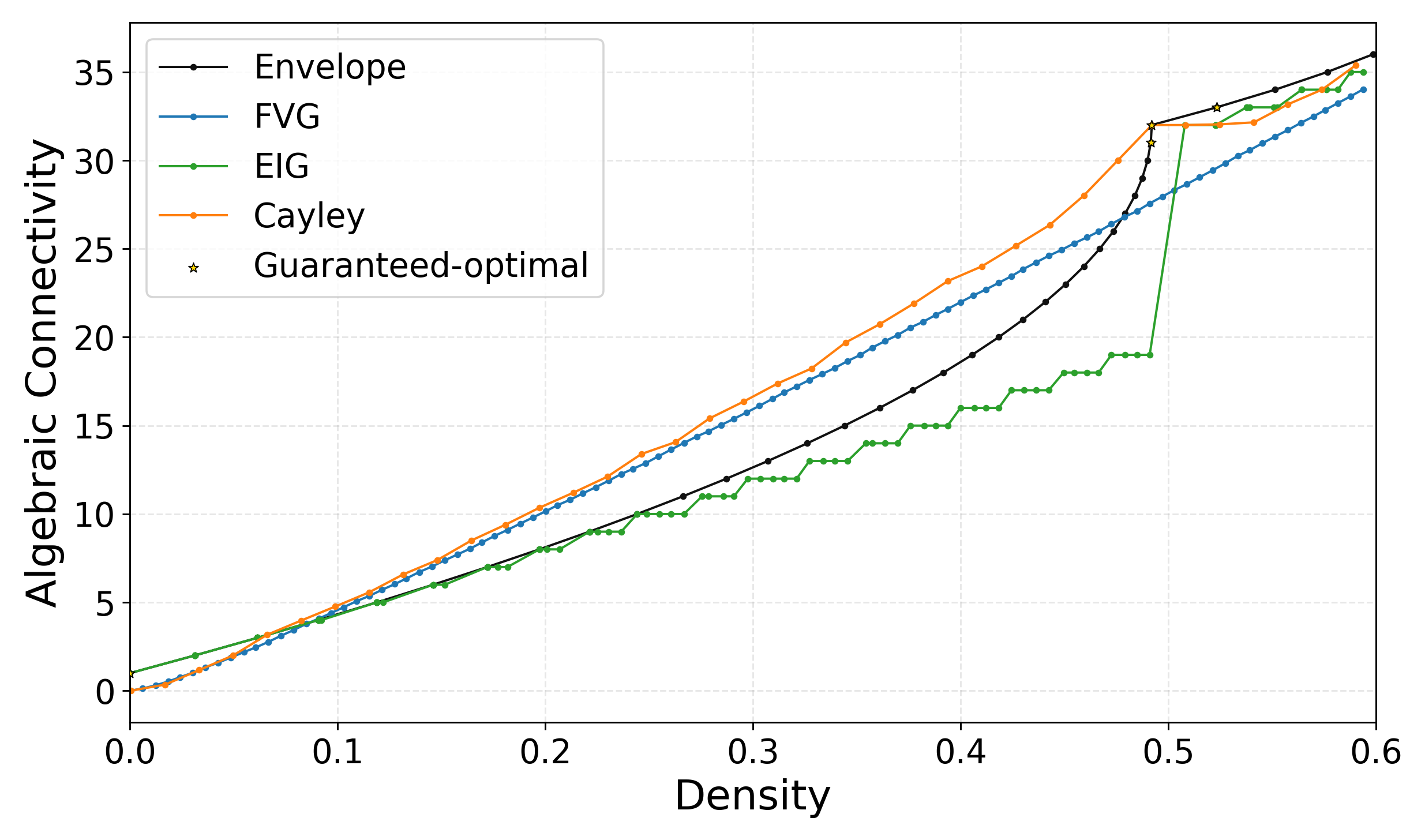}
    \caption{Low-intermediate quality at \(n=64\).}
    \label{fig:backbone_only_lambda2_low}
\end{figure}

\subsection{Backbone+RL Performance}
We compare backbone-initialized RL and FVG on identical \((n,m)\) targets; Figure \ref{fig:backbone_rl_lambda2_n64} and Tables \ref{tab:baseline_lambda2_summary}--\ref{tab:baseline_runtime_summary} summarize quality and runtime.
BB-RL has the largest mean improvement at \(n=16,32,64,128\) and lowest mean cost at \(n=32,64,128\). Reported runtimes are serial, and Cayley sampling admits straightforward parallel speedup, unlike sequential FVG.

\begin{figure}[t]
    \centering
    \includegraphics[width=0.95\linewidth]{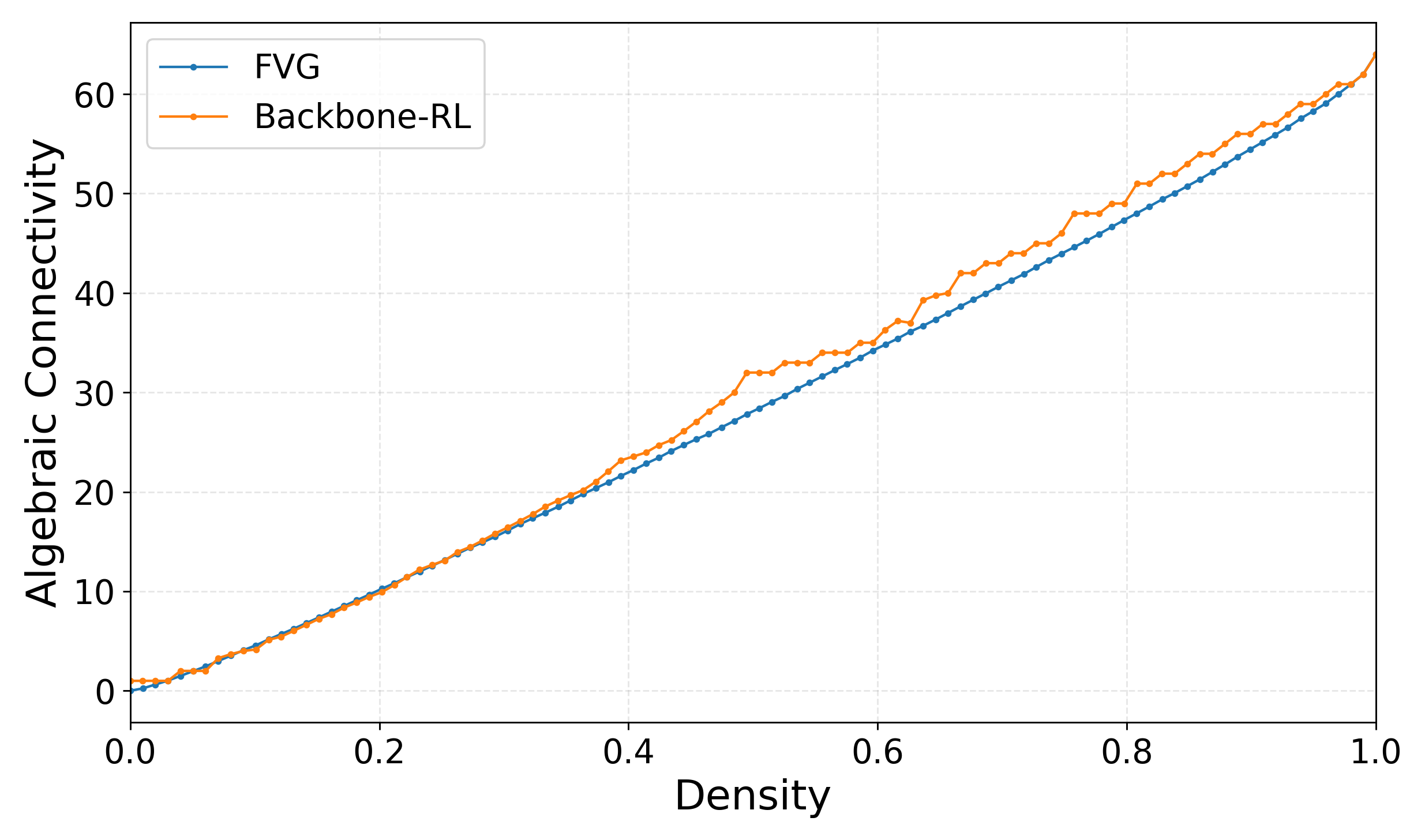}
    \caption{BB-RL vs FVG at \(n=64\).}
    \label{fig:backbone_rl_lambda2_n64}
\end{figure}

\section{Conclusion}
This paper presented a density-dependent framework for fixed-\((n,m)\) algebraic-connectivity maximization that separates structural initialization from completion.
A high-quality low-cost backbone plus RL completion gives the best non-exact quality/runtime scaling among tested baselines. As \(n\) grows, cheaper heuristics lose quality while stronger optimization-based methods become increasingly expensive, supporting backbone-first completion as a practical direction for scalability. Open directions include other cheaply constructible optimal graph families to densify backbone coverage, principled group/subgroup selection for Cayley backbones, characterizing how density-dependent behavior varies with integer structure of \(n\).


\bibliographystyle{unsrt}
\bibliography{susie,jicareer,consensus,push,related_work}

@article{vicsekmodel,
    author    = "A. Jadbabaie and J. Lin and A.S. Morse",
    title     = "Coordination of Groups of Mobile Autonomous Agents Using Nearest Neighbor Rules",
    journal   = "IEEE Transactions on Automatic Control",
    year      = "2003",
    volume    = "48",
    number    = "6",
    pages     = "988-1001",
}

@article{reza1,
    author    = "R. Olfati-Saber and R.M. Murray",
    title     = "Consensus Problems in Networks of Agents with Switching Topology and Time-Delays",
    journal   = "IEEE Transactions on Automatic Control",
    volume    = "49",
    number    = "9",
    pages     = "1520-1533",
    year      = "2004"
}

@article{ReBe05,
    author    = "W. Ren and R.W. Beard",
    title     = "Consensus Seeking in Multiagent Systems Under Dynamically Changing Interaction Topologies",
    journal   = "IEEE Transactions on Automatic Control",
    volume    = "50",
    number    = "5",
    pages     = "655-661",
    year      = "2005"
}

@article{survey,
    author    = "R. Olfati-Saber and J.A. Fax and R.M. Murray",
    title     = "Consensus and Cooperation in Networked Multi-Agent Systems",
    journal   = "Proceedings of the IEEE",
    year      = "2007",
    volume    = "95",
    number    = "1",
    pages     = "215-233"
}

@article{tacle,
    title     = "A Distributed Algorithm for Solving a Linear Algebraic Equation",
    author    = "S. Mou and J. Liu and A.S. Morse",
    journal   = "IEEE Transactions on Automatic Control",
    year      = "2015",
    volume    = "60",
    number    = "11",
    pages     = "2863-2878"
}

@article{split,
    author    = "L. Wang and J. Liu and B.D.O. Anderson and A.S. Morse",
    title     = "Split-spectrum based distributed state estimation for linear systems",
    journal   = "Automatica",
    year      = "2024",
    volume    = "161",
    pages     = "111421"
}

@inproceedings{push,
    author    = "D. Kempe and A. Dobra and J. Gehrke",
    title     = "Gossip-Based Computation of Aggregate Information",
    year      = "2003",
    booktitle = "Proceedings of the 44th IEEE Symposium on Foundations of Computer Science",
    pages     = "482-491"
}

@article{nedic2009distributed,
  title={Distributed subgradient methods for multi-agent optimization},
  author={Nedi\'c, A. and Ozdaglar, A.},
  journal={IEEE Transactions on Automatic Control},
  volume={54},
  number={1},
  pages={48--61},
  year={2009},
}

@book{godsil_royle_agt_2001,
  author    = {Chris Godsil and Gordon Royle},
  title     = {Algebraic Graph Theory},
  publisher = {Springer},
  year      = {2001},
  series    = {Graduate Texts in Mathematics},
  volume    = {207},
  doi       = {10.1007/978-1-4613-0163-9}
}

@incollection{Mohar91,
  author    = {B. Mohar},
  title     = {The Laplacian Spectrum of Graphs},
  booktitle = {Graph Theory, Combinatorics, and Applications},
  publisher = {Wiley},
  address   = {New York, NY, USA},
  year      = {1991},
  pages     = {871--898}
}

@article{Kolokolnikov15,
  author  = {T. Kolokolnikov},
  title   = {Maximizing algebraic connectivity for certain families of graphs},
  journal = {Linear Algebra and its Applications},
  volume  = {471},
  pages   = {122--140},
  year    = {2015}
}

@inproceedings{ishii2016extensions,
  author    = {Ryoya Ishii and Norikazu Takahashi},
  title     = {Extensions of a Theorem on Algebraic Connectivity Maximizing Graphs},
  booktitle = {Proceedings of the 2016 International Symposium on Nonlinear Theory and Its Applications (NOLTA2016)},
  pages     = {598--601},
  year      = {2016},
  month     = nov
}

@article{babai79spectra,
  title   = {Spectra of Cayley graphs},
  author  = {Babai, L{\'a}szl{\'o}},
  journal = {Journal of Combinatorial Theory, Series B},
  volume  = {27},
  number  = {2},
  pages   = {180--189},
  year    = {1979},
  issn    = {0095-8956},
  doi     = {10.1016/0095-8956(79)90079-0}
}

@article{klee_stamps_2022_rank1,
  author  = {Steven Klee and Matthew T. Stamps},
  title   = {Eigenvalues of Graph Laplacians Via Rank-One Perturbations},
  journal = {The Quarterly Journal of Mathematics},
  volume  = {73},
  number  = {2},
  pages   = {609--616},
  year    = {2022},
  doi     = {10.1093/qmath/haab045}
}

@inproceedings{ghosh2006growing,
  title={Growing Well-Connected Graphs},
  author={A. Ghosh and S. Boyd},
  booktitle={Proceedings of the 45th IEEE Conference on Decision and Control},
  pages={6605--6611},
  year={2006},
}

@article{wei2014algebraic,
  title={Algebraic connectivity maximization of an air transportation network: the flight routes' addition/deletion problem},
  author={P. Wei and L. Chen and D. Sun},
  journal={Transportation Research Part E: Logistics and Transportation Review},
  volume={61},
  pages={13--27},
  year={2014},
  doi={10.1016/j.tre.2013.10.008}
}

@inproceedings{somisetty2024spectral,
  title={Spectral Graph Theoretic Methods for Enhancing Network Robustness in Robot Localization},
  author={N. Somisetty and H. Nagarajan and S. Darbha},
  booktitle={Proceedings of the 63rd IEEE Conference on Decision and Control},
  year={2024},
  pages={8811--8818},
}

@article{somisetty2025optimal,
  title={Optimal Robust Network Design: formulations and Algorithms for Maximizing Algebraic Connectivity},
  author={N. Somisetty and H. Nagarajan and S. Darbha},
  journal={IEEE Transactions on Control of Network Systems},
  year={2025},
  volume = {12},
  number = {1},
  pages = {918--929},
  doi={10.1287/ijoc.2023.0320}
}

@ARTICLE{8412478,
  author={Li, Gang and Hao, Zhi Feng and Huang, Han and Wei, Hang},
  journal={IEEE Access}, 
  title={Maximizing Algebraic Connectivity via Minimum Degree and Maximum Distance}, 
  year={2018},
  volume={6},
  number={},
  pages={41249-41255},
  doi={10.1109/ACCESS.2018.2857411}}

@inproceedings{kurahashi2024depth,
  title={Depth-First Search Algorithms for Algebraic Connectivity Maximizing Regular Graphs},
  author={M. Kurahashi and N. Salaani and T. Migita and N. Takahashi},
  booktitle={Proceedings of the 12th International Symposium on Computing and Networking},
  pages={156--161},
  year={2024},
  doi={10.1109/CANDAR63185.2024.00034}
}

@article{kim2010bisection,
  title={Bisection Algorithm of Increasing Algebraic Connectivity by Adding an Edge},
  author={Y. Kim},
  journal={IEEE Transactions on Automatic Control},
  volume={55},
  number={1},
  pages={170--174},
  year={2010},
}

@inproceedings{yin2024accelerating,
  author = {M. Zhou and G. Liu and K. Lu and R. Mao and H. Liao},
  title = {Accelerating the Decentralized Federated Learning via Manipulating Edges},
  year = {2024},
  isbn = {9798400701719},
  url = {https://doi.org/10.1145/3589334.3645509},
  doi = {10.1145/3589334.3645509},
  booktitle = {Proceedings of the 2024 ACM Web Conference},
  pages = {2945--2954},
  numpages = {12},
}

@INPROCEEDINGS{9981584,
  author={K.J. Doherty and D.M. Rosen and J.J. Leonard},
  booktitle={Proceedings of the 2022 IEEE/RSJ International Conference on Intelligent Robots and Systems}, 
  title={Spectral Measurement Sparsification for Pose-Graph {SLAM}}, 
  year={2022},
  volume={},
  number={},
  pages={1--8},
  doi={10.1109/IROS47612.2022.9981584}}

@article{schulman2017ppo,
  title={Proximal Policy Optimization Algorithms},
  author={Schulman, John and Wolski, Filip and Dhariwal, Prafulla and Radford, Alec and Klimov, Oleg},
  journal={arXiv preprint arXiv:1707.06347},
  year={2017}
}

@misc{zhu2026code,
  author={Z. Zhu},
  title={Algebraic Connectivity Optimization},
  howpublished={GitHub repository. \url{https://github.com/ZeruZhu/Algebraic-Connectivity-Optimization-ACC2027}},
  year={2026}  
}

@article{MOSKAOYAMA2008677,
  title = {Maximum algebraic connectivity augmentation is {NP}-hard},
  journal = {Operations Research Letters},
  volume = {36},
  number = {6},
  pages = {677--679},
  year = {2008},
  issn = {0167-6377},
  doi = {https://doi.org10.1016/j.orl.2008.09.001},
  url = {https://www.sciencedirect.com/science/article/pii/S0167637708001077},
  author = {D. Mosk-Aoyama},
}

@article{Shahbaz23,
  title="{Algebraic connectivity: local and global maximizer graphs}",
  author="K. Shahbaz and M.N. Belur and A. Ganesh",
  journal={IEEE Transactions on Network Science and Engineering},
  volume={10},
  number={3},
  pages="1636--1647",
  year={2023}
}

@article{Fiedler73,
  author =       "M. Fiedler",
  title =        "{Algebraic connectivity of graphs}",
  journal =      "Czechoslovak Mathematical Journal",
  year =         "1973",
  volume = "23",
  number = "2",
  pages = "298--305"
}

@article{ogiwara,
    author    = "K. Ogiwara and T. Fukami and N. Takahashi",
    title     = "Maximizing Algebraic Connectivity in the Space of Graphs With a Fixed Number of Vertices and Edges",
    journal   = "IEEE Transactions on Control of Network Systems",
    year      = "2017",
    volume    = "4",
    number    = "2",
    pages     = "359-368"
}

@inproceedings{ming,
    author    = "M. Cao and C.W. Wu",
    title     = "Topology design for fast convergence of network consensus algorithms",
    booktitle = "Proceedings of the 2007 IEEE International Symposium on Circuits and Systems",
    year      = "2007",
    pages      = "1029-1032"
}

@inproceedings{moura,
    author    = "S. Kar and J.M.F. Moura",
    title     = "Topology for Global Average Consensus",
    booktitle = "Proceedings of the 40th Asilomar Conference on Signals, Systems and Computers",
    year      = "2006",
    pages      = "276-280"
}

@inproceedings{bayen,
    author    = "M. Rafiee and A.M. Bayen",
    title     = "Optimal Network Topology Design in Multi-Agent Systems for Efficient Average Consensus",
    booktitle = "Proceedings of the 49th IEEE Conference on Decision and Control",
    year      = "2010",
    pages      = "3877-3883"
}

@inproceedings{dai,
    author    = "R. Dai and M. Mesbahi",
    title     = "Optimal topology design for dynamic networks",
    booktitle = "Proceedings of the 50th IEEE Conference on Decision and Control",
    year      = "2011",
    pages      = "1280-1285"
}

@inproceedings{dong,
    author    = "D. Xue and A. Gusrialdi and S. Hirche",
    title     = "A Distributed Strategy for Near-Optimal Network Topology Design",
    booktitle = "Proceedings of the 21st International Symposium on Mathematical Theory of Networks and Systems",
    year      = "2014",
    pages      = "7-14"
}

\end{document}